\documentclass[12pt,reqno]{amsart}
\usepackage{amsmath}
\usepackage{amssymb}
\usepackage{amstext}
\usepackage{a4wide}
\usepackage{bm}
\usepackage{graphicx}
\usepackage[numbers,sort&compress]{natbib}
\allowdisplaybreaks \numberwithin{equation}{section}
\usepackage{color}
\usepackage{cases}

\numberwithin{equation}{section}

\newtheorem{theorem}{Theorem}[section]
\newtheorem{proposition}[theorem]{Proposition}

\newtheorem{lemma}[theorem]{Lemma}

\theoremstyle{definition}

\theoremstyle{remark}
\newtheorem{remark}[theorem]{Remark}

\begin{document}

\title[Stable plane Euler flows with concentrated and sign-changing vorticity]{Stable plane Euler flows  with concentrated and sign-changing vorticity}

 \author{Guodong Wang, Bijun Zuo}
\address{Institute for Advanced Study in Mathematics, Harbin Institute of Technology, Harbin 150001, P.R. China}
\email{wangguodong@hit.edu.cn}
 \address{College of Mathematical Sciences, Harbin Engineering University, Harbin {\rm150001}, PR China}
\email{bjzuo@amss.ac.cn}


\begin{abstract}
We construct a family of steady solutions to the two-dimensional incompressible Euler equation in a general bounded domain, such that the vorticity is supported in two well-separated regions of small diameter and converges to a pair of point vortices with opposite signs. Compared with previous results,  we do not need to assume the existence of an isolated local minimum point of the Kirchhoff-Routh function. Moreover, due to their variational nature, the solutions obtained are Lyapunov stable in $L^p$ norm of the vorticity.
The proofs are achieved by maximizing the kinetic energy over an appropriate family of rearrangement classes of sign-changing functions and studying the limiting behavior of the maximizers.

 \end{abstract}

\maketitle
\section{Introduction}

 The evolution of an inviscid homogeneous fluid in a  domain $D\subset\mathbb R^2$  is driven by the two-dimensional Euler equation, which in vorticity formulation can be written as follows:
\begin{equation}\label{ve1}
\partial_{t}\omega+\bm v\cdot\nabla\omega=0,\quad t>0,\,\,\bm x=(x_{1},x_{2})\in D,
\end{equation}
where  $\omega=\omega(t,\bm x)\in\mathbb R$ is the scalar vorticity  and $\bm v=\bm v(t,\bm x)\in\mathbb R^{2}$  is the velocity. In most situations, $\bm v$ can be recovered from $\omega$ via the Biot-Savart law.  For example, when  $D$ is bounded and simply-connected, and $\bm v$ is everywhere tangential on the boundary, the Biot-Savart law can be expressed in terms of the Green operator $\mathcal G$ as follows:
\begin{equation}\label{bsl}
\bm v=\nabla^{\perp}\mathcal G\omega:=(\partial_{x_{2}}\mathcal G\omega,-\partial_{x_{1}}\mathcal G\omega),
\end{equation}
where $ \perp$ denotes clockwise rotation through $\pi/2$; see \cite{MPu}, Chapter 1 for example.
In the literature, there are many global existence results for  equation \eqref{ve1} with initial vorticity in various function spaces; see \cite{De,DM,MB,MPu,V,Wo,Y}.

The evolution of vorticity by  equation \eqref{ve1} can be very complicated. However, when the vorticity is sufficiently concentrated in a finite number of small regions,  equation \eqref{ve1} can be well approximated by a much simpler ODE system, called the \emph{point vortex system}. The point vortex system was first introduced by Helmholtz \cite{Hel}, and later studied by Kirchhoff \cite{Kir}, Routh \cite{Rou}, and Lin \cite{Lin}. 
Roughly speaking, the point vortex system says that  $k$ separated blobs of concentrated vorticity evolve like $k$ individual particles, the locations $\bm x_{1},\cdot\cdot\cdot,\bm x_{k}$ of which are determined by the following ODE system:
\begin{equation}\label{kre}
\kappa_{i}\frac{d\bm x_{i}(t)}{dt}=\nabla_{\bm x_{i}}^{\perp} W_{\bm\kappa}(\bm x_{1},\cdot\cdot\cdot,\bm x_{k}),\quad i=1,\cdot\cdot\cdot,k,
\end{equation}
where  $\kappa_{i}$ is the integral of the $k$-th blob of vorticity, and 
 $W_{\bm \kappa}$ is  the Kirchhoff-Routh function related to $\bm\kappa=(\kappa_{1},\cdot\cdot\cdot,\kappa_{k})$, the precise definition of which is given by \eqref{krf} in Section 2. Such a approximation is called the \emph{desingularization of point vortices} and has been verified on a rigorous level by many authors; see \cite{Cap, CS1, CS2, Don, MP1,MP2,MP3,T3}.

Except for the evolutionary case, an related problem of importance is the \emph{steady  desingularization problem}, i.e., to construct a family of steady Euler flows such that the vorticity is concentrated around a finite number of points, which form an equilibrium of the point vortex system. Here by a steady Euler flow, we mean that its vorticity $\zeta$ satisfies 
 \begin{equation}\label{ssol}
\nabla^{\perp}\mathcal G\zeta\cdot\nabla\zeta=0.
 \end{equation}
When $\zeta$ is not differentiable in the classical sense, we need to interpret \eqref{ssol} in the following weak sense:
 \begin{equation}\label{ssol}
\int_{D}\zeta(\bm x)\nabla^{\perp}\mathcal G\zeta(\bm x)\cdot\nabla\varphi(\bm x) d\bm x=0,\quad\forall\,\varphi\in C_{c}^{\infty}(D).
 \end{equation}
Note that by the Sobolev embedding theorem and standard elliptic estimates, the integral in \eqref{ssol} makes sense for any $\zeta\in L^{4/3}(D)$.
In addition to existence, an interesting  related problem is to study the stability of these steady flows, especially when the equilibrium of the corresponding point vortex system is stable.

The steady  desingularization problem was studied for the first time by Turkington \cite{T12} via a variational approach. Turkington considered the following maximization problem:
 \begin{equation}\label{xpt}
\sup_{v\in\mathcal K_{\lambda}}E(v),
\end{equation}
where $\mathcal K_{\lambda}$ is a set of functions in $D$ parametrized by a large positive number $\lambda$, 
\begin{equation}\label{keps}
\mathcal K_{\lambda}=\left\{v\in L^{\infty}(D)\mid 0\leq v\leq \lambda \mbox{ a.e. in }D,\,\,\int_{D}vd\bm x=1\right\},
\end{equation}
and $E$ the kinetic energy defined by
\begin{equation}\label{kine}
E(v)=\int_{D}v(\bm x)\mathcal Gv (\bm x)d\bm x.
\end{equation}
Turkington proved the existence of a maximizer, and showed that any maximizer $\zeta$ has a patch form, i.e., $\zeta=\lambda\bm 1_{A_{\lambda}}$ for some unknown open set $A_{\lambda}$ depending on $\lambda$, where $\bm 1_{A_{\lambda}}$ denotes the characterization function of $A_\lambda$. More importantly, Turkington established fine asymptotic estimates for $A_\lambda$ as $\lambda\to+\infty$ based on the energy expansion method, showing that $A_{\lambda}$ ``shrinks'' to some global minimum point of the Robin function $H$ of the domain $D$ (defined by \eqref{derob} in Section 2)  as $\lambda\to+\infty$. 

Turkington's result was later extended by 
Elcrat and Miller \cite{EM2}. The extension is twofold: First, the steady flows they obtained have \emph{multiple} concentrated vorticity components. To achieve this, they imposed in the variational problem \eqref{xpt} an additional constraint that the vorticity is supported near a finite number of given points, which constitute an isolated local minimum point of the Kirchhoff-Routh function. Second, the steady vortex flows they obtained have general vorticity profiles, no longer limited to the patch case. To this end, Elcrat and Miller applied the variational principle on classes of rearrangements established by Burton \cite{B1,B2}. 

The above vorticity method developed by Turkington and Elcrat-Miller has proved to be an effective tool in the study of the steady desingularization problem as well as other related problems. See \cite{CQZZ,CWCV,CWWZ,CWZ,Dek} and the references therein.
Except for the vorticity method, 
an alternative approach to study the steady desingularization problem is to solve a certain semilinear elliptic equation with the stream function as the unknown function, which is usually called the stream function method; see  \cite{CLW,CPY,SV} and the references therein. 
When applying the stream function method, one usually obtains a steady flow with finer estimates for the   stream function, but less information on the vorticity, which makes it hard to analyze stability.

A key assumption in Elcrat-Miller's paper \cite{EM2}, as well as some similar papers  such as \cite{CWZ,CWCV}, is the existence of an isolated local minimum point of the Kirchhoff-Routh function $W_{\bm\kappa}$. However, as far as we know, there is no general result to guarantee the existence of such a point for a general bounded domain in the literature. We summarize some known existence and nonexistence results as follows:
\begin{itemize}
\item[(i)] If $k=1$ and $D$ is convex, then $W_{\bm\kappa}$ equals some positive multiple of the Robin function $H$.  By Caffarelli and Friedman \cite{CF}, $H$ is strictly convex in $D$, hence $W_{\bm\kappa}$  has a unique global minimum point in this case.
\item[(ii)] If $k\geq 2$, $\kappa_{1},\cdot\cdot\cdot,\kappa_{k}$ are all positive, and $D$ is convex, then by Grossi and Takahashi \cite{GT},  $W_{\bm\kappa}$ has no critical point, and thus has no isolated local minimum point. 
\item[(iii)] If $k=2$ and $\kappa_{1}\kappa_{2}<0$, then it is easy to check that $W_{\bm\kappa}$ attains its global minimum value in $D\times D$; however, it is not clear whether this global minimum point is isolated. If $D$ is a disk, then  every local minimum point in this case can not be isolated due to rotational symmetry.
 \end{itemize}
For the above reason,  it is not clear whether  a steady Euler flow  with concentrated and sign-changing vorticity exists for a general bounded domain, especially when there is no  isolated local minimum point for the Kirchhoff-Routh function.

 Our purpose in this paper is to construct a family of steady Euler flows with two concentrated vorticity components of opposite signs without any requirement on the geometry of the domain.
 Our strategy is as follows.
According to the variational principle for strictly convex functionals on rearrangement classes established by Burton \cite{B1}, a steady flow can be obtained as a maximizer of the kinetic energy  subject to the constraint that the vorticity is the rearrangement of a given function $v$. Then we choose a family of sign-changing functions $v$ and 
analyze the asymptotic properties of the maximizers as the measures of $\{v^{+}>0\}$, $\{v^{-}>0\}$ tend to zero and the integrals of  $v^{+}$, $-v^{-}$ tend to   $\kappa_{1}$, $\kappa_{2}$,  with  $\kappa_{1}>0$, $\kappa_{2}<0$ being prescribed.
 Under some reasonable assumptions, we can prove that for any maximizer $\zeta$,  the supports of $\zeta^{+},$ $\zeta^{-}$  ``shrink'' to two different points, the locations of which are totally determined by $\kappa_{1}$, $\kappa_{2}$ and the geometry of $D$.
In this way,  we provide a kind of specific desingularization for a steady point vortex pair with opposite signs. 

Our asymptotic analysis basically follows Turkington's energy method in \cite{T12}. Opposed to the situation of nonnegative rearrangements in Turkington's paper, two essential difficulties appear when considering sign-changing rearrangements. First, we need to deal with the interaction energy between the positive  and the negative vorticity components. Note that in Elcrat-Miller's paper \cite{EM2}, although there are multiple vorticity components, their interaction energies are all bounded due to the support constraint. Second,  the estimates for the Lagrangian multipliers in this paper is more complicated since the supports of the positive and the negative vorticity components may not be well separated. For example, uniform boundedness from below (above) for  the positive (negative) Lagrangian multiplier is no longer obvious as in \cite{EM2}, but requires careful treatment.

A notable feature of the steady flows we obtain is that they are Lyapunov stable in $L^p$ norm of the vorticity, which is mainly due to their variational nature. To our knowledge, there are very few examples on stable plane Euler flows  with concentrated vorticity in the literature. See \cite{CWCV,CWN,Wang} for more detailed discussions.

It is worth mentioning that our method also applies,  without any essential difficulty, to the steady desingularization problem in a multiply-connected bounded domain. The main difference lies in the Biot-Savart law, which takes a more complicated form in the multiply-connected case (see Appendix C in \cite{WZ} for example).

The rest of this paper is organized as follows. In Section 2, we present the rigorous mathematical setting and state our main results, i.e., Theorems \ref{thm1} and \ref{thm2}. In Section 3, we list some preliminaries lemmas that are used in the proofs.  Sections 4 and 5 are devoted to the proofs of Theorems \ref{thm1} and \ref{thm2}. 

\section{Main results}

In this section, we give the rigorous formulation of the problem and state the main results.

To begin with, we specify some notation and definitions that are used in the rest of this paper (although some of them have already  appeared in Section 1).
Let $D$ be  a smooth, bounded and simply-connected domain in $\mathbb R^{2}$.
Let $G$ be the Green function of $-\Delta$ in $D$ with zero Dirichlet boundary condition. Denote by $h$ the regular part of $G$, i.e.,
\[h(\bm x,\bm y)=-\frac{1}{2}\ln|\bm x-\bm y|-G(\bm x,\bm y).\]
 Note that $h$ is smooth and bounded from below in $D\times D.$ The Robin function $H$ of the domain $D$ is defined by
\begin{equation}\label{derob}
H(\bm x):=h(\bm x,\bm x),\quad \bm x\in D.
\end{equation} 
The Green operator $\mathcal G$ is  the inverse of $-\Delta$ in $D$ with zero Dirichlet boundary condition, which in terms of the Green function can be expressed as follows:
\[\mathcal Gv(\bm x)=\int_{D}G(\bm x,\bm y)v(y)d\bm y,\quad \bm x\in D.\]
Note that $1<p<+\infty$, $\mathcal G$ is a bounded linear bijective map from $L^p(D)$ onto $W^{2,p}\cap W^{1,p}_0(D)$.

 For $\bm\kappa=(\kappa_{1},\cdot\cdot\cdot,\kappa_{k})\in\mathbb R^{k}$ such that $\kappa_{1}, \cdot\cdot\cdot,\kappa_{k}\neq 0$, define the Kirchhoff-Routh function  $W_{\bm\kappa}$ related to $\bm\kappa$ as follows:
\begin{equation}\label{krf}
W_{\bm\kappa}(\bm x_{1},\cdot\cdot\cdot,\bm x_{k})=-\sum_{1\leq i<j\leq k}\kappa_{i}\kappa_{j}G(\bm x_{i},\bm x_{j})+\frac{1}{2}\sum_{i=1}^{k}\kappa_{i}^{2}H(\bm x_{i}),\,\,\,\mbox{$\bm x_{i}\in D,$ $\bm x_{i}\neq \bm x_{j}$ if $i\neq j$.}
\end{equation}
In particular, if  $k=1$, then $W_{\bm \kappa}$ is a positive multiple of $H$; if $\bm\kappa=(\kappa_{1},\kappa_{2})\in\mathbb R^{2}$, then $W_{\bm \kappa}$ has the form
\begin{equation}\label{krf2}
W_{\bm\kappa}(\bm x_{1},\bm x_{2})=-\kappa_{1}\kappa_{2}G(\bm x_{1},\bm x_{2})+\frac{1}{2}\kappa_{1}^{2}H(\bm x_{1})+\frac{1}{2}\kappa_{2}^{2}H(\bm x_{2}),\quad\mbox{$\bm x_{1},\bm x_2\in D,$ $\bm x_{1}\neq \bm x_{2}$.}
\end{equation}

 For a Lebesgue measurable function $v: D\to\mathbb R$,  the class of rearrangements of $v$ is defined as the set of all measurable functions $u:D\to\mathbb R$ such that
\[  \mathfrak m(\{\bm x\in D\mid u(\bm x)>s\})=\mathfrak m(\{\bm x\in D\mid v(\bm x)>s\})\quad\forall\,s\in\mathbb R,\]
where  $\mathfrak m(A)$ is the two-dimensional Lebesgue measure. 

Recall   the definition of the  kinetic energy $E$  by \eqref{kine} in Section 1.
The following theorem is about the maximization of $E$ relative to a given  rearrangement class of a fixed $L^p$ function, which is a straightforward corollary of the results in \cite{B1,B4,CWP,Wang}. 

 \begin{theorem}[\cite{B1,B4,CWP,Wang}]\label{thm0}
Let $1<p<+\infty$. Let $\mathcal R$ be the class of rearrangements of some function in $L^{p}(D)$. 
Denote by $\mathcal M$ the set of maximizers of $E$ relative to $\mathcal R$,
\[\mathcal M=\{v\in\mathcal R\mid E(v)=M\},\quad M:=\sup_{v\in\mathcal R}E(v).\]
 Then the following assertions hold:
\begin{itemize}
\item[(i)] $\mathcal M$ is  nonempty and compact  in $L^{p}(D)$.
\item[(ii)] For any $\zeta\in\mathcal M,$ there exists some increasing function $\phi:\mathbb R\to\mathbb R\cup\{\pm\infty\}$ such that
\begin{equation}\label{incp}
\zeta=\phi(\mathcal G\zeta)\quad\mbox{\rm a.e. in }\,\,D.
\end{equation}
\item[(iii)]  $\mathcal M$ is Lyapunov stable in the following sense: for any  $\epsilon>0$, there exists some $\delta>0$, such that for any smooth solution $\omega(t,\bm x)$ of the Euler equation \eqref{ve1},  
\[\min_{v\in\mathcal M}\|\omega(0,\cdot)-v\|_{L^{p}(D)}<\delta\Longrightarrow \sup_{t>0}\min_{v\in\mathcal M}\|\omega(t,\cdot)-v\|_{L^{p}(D)}<\epsilon.\]
\item[(iv)]  If additionally $4/3\leq p<+\infty,$  then any $\zeta\in\mathcal M$ satisfies \eqref{ssol}. 
\end{itemize}
 \end{theorem}
 
Items (i) and (ii) in  Theorem \ref{thm0} follow from Theorem 7 and Corollary 2 in \cite{B1}; item (iii) follows from Theorem 3.1 in \cite{Wang}; item (iv) follows from Lemma 6 in \cite{B4} or  Theorem 1.2 in \cite{CWP}.

To achieve our main purpose in this paper, i.e., to obtain steady Euler flows with two concentrated vorticity components of opposite signs, 
we consider the maximization of $E$ relative to an appropriate family of rearrangement classes and study the asymptotic behavior of the maximizers.

Let  $\mathcal R_{\bm\varepsilon}\subset L^p(D),$ $1<p<+\infty$, be a rearrangement class  parametrized by a two-dimensional vector $\bm{\varepsilon}=(\varepsilon_{1},\varepsilon_{2})$ with $0<\varepsilon_{1}, \varepsilon_{2}<<1$, such that 
\begin{equation}\label{ep0}
\mathfrak m(\{ \bm x\in D\mid v(\bm x) >0\})= \pi\varepsilon_{1}^{2}, \quad \mathfrak m(\{ \bm x\in D\mid v(\bm x) <0\})= \pi\varepsilon_{2}^{2},\quad \forall\, v\in\mathcal R_{\bm\varepsilon}.
\end{equation}
We make the following assumptions on $\mathcal R_{\bm\varepsilon}:$
\begin{itemize}
\item[(H1)]  For arbitrarily chosen $v_{\bm\varepsilon}\in\mathcal R_{\bm\varepsilon},$
\begin{equation}\label{ep1}
\lim_{|\bm\varepsilon|\to 0}\|v_{\bm\varepsilon}^{+}\|_{L^{1}(D)}= \kappa_{1}, \quad \lim_{|\bm\varepsilon|\to 0}\|v_{\bm\varepsilon}^{-}\|_{L^{1}(D)}= -\kappa_{2},
\end{equation}
where   $\kappa_{1}>0,\kappa_{2}<0$ are prescribed;

\item[(H2)]  For arbitrarily chosen $v_{\bm\varepsilon}\in \mathcal R_{\bm\varepsilon}$, 
\begin{equation}\label{ep2}
\limsup_{|\bm\varepsilon|\to 0}\varepsilon_{1}^{2/p'}\|v_{\bm\varepsilon}^{+}\|_{L^{p}(D)}<+\infty,\quad  \limsup_{|\bm\varepsilon|\to 0}\varepsilon_{2}^{2/p'}\|v_{\bm\varepsilon}^{-}\|_{L^{p}(D)}<+\infty,
\end{equation}
where $p'=p/(p-1)$ is the conjugate exponent of $p$.
\end{itemize}
Note that the above assumptions on $\mathcal R_{\bm\varepsilon}$ are weaker than those in \cite{EM2,T12}.

 We consider the maximization problem
\begin{equation}\label{maxep}
M_{\bm\varepsilon}=\sup_{v\in  \mathcal R_{\bm\varepsilon}}E(v).
\end{equation}
By Theorem \ref{thm0}, the set of maximizers $\mathcal M_{\bm\varepsilon}$  for  \eqref{maxep} is not empty and satisfies (i)-(iv) in Theorem \ref{thm0}.
 For any $\zeta\in\mathcal M_{\bm\varepsilon},$ define the positive vortex core $V^+_\zeta$ and the negative vortex core $V^-_\zeta$ related to $\zeta$ as follows:
 \begin{equation}\label{defvc}
 V^+_\zeta:= \{\bm x\in D\mid \zeta(\bm x)>0\},\quad V^-_\zeta:= \{\bm x\in D\mid \zeta(\bm x)<0\}.
 \end{equation}
 
Our first result is the following theorem.
\begin{theorem}\label{thm1}
Let $1<p<+\infty,$  $\kappa_1>0,$ $\kappa_2<0$ be given, and $\bm{\varepsilon}=(\varepsilon_{1},\varepsilon_{2})$ be a parameter vector with $0<\varepsilon_{1}, \varepsilon_{2}<<1$. Let $\{\mathcal R_{\bm\varepsilon}\}\subset L^p(D)$   be a family of rearrangement classes such that \eqref{ep0}-\eqref{ep2} hold. Let $\mathcal M_{\bm\varepsilon}$ be the set of maximizers of \eqref{maxep}. Then the following assertions hold:
\begin{itemize}
\item [(i)](Size of vortex cores) There exists  $C>0$, not depending  on $\bm\varepsilon$,  such that   
\[\mbox{\rm diam(}V^+_\zeta)\leq C\varepsilon_{1},\quad\mbox{\rm diam(}V^-_\zeta)\leq C\varepsilon_{2},\quad \forall\,\zeta\in\mathcal R_{\bm\varepsilon},\]
where $V^\pm_\zeta$ is defined by \eqref{defvc} and $\mbox{diam}(V^\pm_\zeta)$ is the diameter of $V^\pm_\zeta$.
\item[(ii)](Limiting location of vortex cores) 
For arbitrarily chosen $\zeta_{\bm\varepsilon}\in\mathcal R_{\bm\varepsilon},$ 
 suppose up to a subsequence
\begin{equation}\label{pp2}
\mathbf X(\zeta_{\bm\varepsilon}^{+})\to\bar{\bm x}_{1}\in \bar D,\,\,\mathbf X(\zeta_{\bm\varepsilon}^{-})\to\bar{\bm x}_{2}\in \bar D \quad\mbox{as } |\bm\varepsilon|\to0.
\end{equation}
where $\mathbf X(\zeta_{\bm\varepsilon}^{\pm})$ is the center of $\zeta^{\pm}_{\bm\varepsilon}$, 
\begin{equation}\label{dece}
\mathbf X(\zeta_{\bm\varepsilon}^{\pm})=\frac{1}{\int_D \zeta_{\bm\varepsilon}^{\pm}(\bm x) d\bm x}\int_{D}\bm x \zeta_{\bm\varepsilon}^{\pm}(\bm x)d\bm x.
\end{equation}
Then $ \bar{\bm x}_{1},\bar{\bm x}_{2}\in  D$, $ \bar{\bm x}_{1}\neq \bar{\bm x}_{2}$, and  $(\bar{\bm x}_{1},\bar{\bm x}_{2})$ is a global minimum point of the Kirchhoff-Routh function $W_{\bm \kappa}$ with $\bm\kappa=(\kappa_{1},\kappa_{2})$.
\end{itemize}
\end{theorem}

\begin{remark}
Note that the assumption (H2) is made just for some technical requirements in the proofs. It is not clear whether Theorem \ref{thm1} holds if (H2) is weakened or removed.
\end{remark}

If $p\geq 4/3,$ then Theorem \ref{thm1} provides a family of Lyapunov stable solutions to the steady Euler equation such that the vorticity is supported in two separate regions of small diameter, approaching a pair of point vortices 
whose locations are completely determined by the geometry of the domain.

With the help of (H2), we can further study the limiting profile of the maximizers.  
For fixed $v\in\mathcal R_{\bm\varepsilon}$, \emph{let $\rho_{\bm\varepsilon,1}, \rho_{\bm\varepsilon,2}$ be the symmetric-decreasing rearrangement of $v^+, v^-$  with respect to the origin $\bm 0$}, i.e., $\rho_{\bm\varepsilon,1}, \rho_{\bm\varepsilon,2}$ are radially symmetric and  nonincreasing functions such that for any $s\in\mathbb R,$
\begin{equation}\label{repp1}
\mathfrak m(\{\bm x\in\mathbb R^{2}\mid \rho_{\bm\varepsilon,1}(\bm x)>s\})=\mathfrak m(\{\bm x\in D\mid v^{+}(\bm x)>s\}),
\end{equation}
\begin{equation}\label{repp2}
\mathfrak m(\{\bm x\in\mathbb R^{2}\mid \rho_{\bm\varepsilon,2}(\bm x)>s\})=\mathfrak m(\{\bm x\in D\mid v^{-}(\bm x)>s\}).
\end{equation}
Note that $\rho_{\bm\varepsilon,1}, \rho_{\bm\varepsilon,2}$ do not depend on the choice of $v$.
Define
\begin{equation}\label{rhss}
\varrho_{\bm\varepsilon,1}(\bm x)=\varepsilon_{1}^{2}\rho_{\bm\varepsilon,i}(\varepsilon_{1}\bm x),\quad \varrho_{\bm\varepsilon,2}(\bm x)=\varepsilon_{2}^{2}\rho_{\bm\varepsilon,i}(\varepsilon_{2}\bm x).
 \end{equation}
Then up to a  set of zero Lebesgue measure,
\begin{equation}\label{etc1}
\{\bm x\in\mathbb R^2\mid \varrho_{\bm\varepsilon,1}(\bm x)>0\}=\{\bm x\in\mathbb R^2\mid \varrho_{\bm\varepsilon,2}(\bm x)>0\}=B_{1}(\bm 0).
\end{equation}
Moreover,   some simple computations show that
\begin{equation}\label{etc2}
\|\varrho_{\bm\varepsilon,1}\|_{L^{1}(\mathbb R^{2})}=\kappa_{\bm\varepsilon,1}, \quad \|\varrho_{\bm\varepsilon,2}\|_{L^{1}(\mathbb R^{2})}=\kappa_{\bm\varepsilon,2},
\end{equation}
\begin{equation}\label{etc3}
\|\varrho_{\bm\varepsilon,1}\|_{L^{p}(\mathbb R^{2})}=\varepsilon_{1}^{2/p'}\|v^{+}\|_{L^{p}(D)},\quad \|\varrho_{\bm\varepsilon,2}\|_{L^{p}(\mathbb R^{2})}=\varepsilon_{2}^{2/p'}\|v^-\|_{L^{p}(D)} \quad (\forall\,v\in\mathcal R_{\bm\varepsilon}).
\end{equation}
By (H2) and \eqref{etc3}, it holds that 
\begin{equation}\label{ep200}
\limsup_{|\bm\varepsilon|\to 0} \|\varrho_{\bm\varepsilon,1}\|_{L^{p}(\mathbb R^{2})}<+\infty,\quad  \limsup_{|\bm\varepsilon|\to 0} \|\varrho_{\bm\varepsilon,2}\|_{L^{p}(\mathbb R^{2})}<+\infty.
\end{equation}
Hence up to a subsequence, $\varrho_{\bm\varepsilon,i}$  converges weakly to some function $\varrho_{i}$ in $L^p(\mathbb R^2)$ as $|\bm\varepsilon|\to0,$ $i=1,2.$  Note that such $\varrho_{1}, \varrho_{2}$ must be radially symmetric and nonincreasing, which can be easily verified by applying Lemma \ref{bgu} in Section 3.

Now we are ready to state our second result.
 
 \begin{theorem}\label{thm2}
For $\zeta_{\bm\varepsilon}\in\mathcal R_{\bm\varepsilon},$ define
\[\xi_{\bm\varepsilon,1}(\bm x)=\varepsilon_{1}^{2}\zeta_{\bm\varepsilon}^{+}\left(\varepsilon_{1}\bm x+\mathbf X(\zeta_{\bm\varepsilon}^{+})\right),\quad \xi_{\bm\varepsilon,2}(\bm x)=\varepsilon_{2}^{2}\zeta_{\bm\varepsilon}^{+}\left(\varepsilon_{2}\bm x+\mathbf X(\zeta_{\bm\varepsilon}^{-})\right).\]
Let $\varrho_{\bm\varepsilon,1}, \varrho_{\bm\varepsilon,2}$ be defined by \eqref{rhss}. 
Then  the following assertions hold:
\begin{itemize}
\item[(i)] Up to a subsequence, if $\varrho_{\bm\varepsilon,i}$ converges weakly to some $\varrho_{i}$ in $L^{p}(\mathbb R^{2})$, then $\xi_{\bm\varepsilon,i}$ converges weakly to $\varrho_{i}$ in $L^{p}(\mathbb R^{2})$, where $i=1,2.$
\item[(ii)]Up to a subsequence, if $\varrho_{\bm\varepsilon,i}$ converges strongly to some $\varrho_{i}$ in $L^{p}(\mathbb R^{2})$, then $\xi_{\bm\varepsilon,i}$ converges strongly  to $\varrho_{i}$ in $L^{p}(\mathbb R^{2})$, where $i=1,2.$

\end{itemize}
\end{theorem}

\section{Preliminaries}

The following lemma will used in Lemmas \ref{bdd0} and \ref{bdd00} in Section 4.
\begin{lemma}\label{le21}
Let $\Omega\subset \mathbb R^{2}$ be a smooth bounded domain and $1<p<+\infty$. Suppose $u\in W^{2,p}\cap W^{1,p}_{0}(\Omega)$ is nonnegative. Then there exists some $C>0$, depending only on $p$ and $\Omega$, such that 
\[\|\nabla u\|_{L^{2}(\Omega)}\leq C\|\Delta u\|_{L^{p}(\Omega)}\mathfrak m^{1/p'}\left(\{\bm x\in\Omega\mid u(\bm x)>0\}\right).\]
\end{lemma}

\begin{proof}
For simplicity, denote 
\[f:=-\Delta u\in L^{p}(\Omega),\quad U:= \{\bm x\in\Omega\mid u(\bm x)>0\}.\] By integration by parts and H\"older's inequality,
\begin{equation}\label{e201}
\int_{\Omega}|\nabla u|^{2}dx=\int_{D}uf dx\leq \|u\|_{L^{p'}(\Omega)}\|f\|_{L^{p}(\Omega)}.
\end{equation}
To complete the proof, it suffices to show that there exists some   $C>0$, depending only on $p$ and $\Omega$, such that 
\begin{equation}\label{e202}
  \|u\|_{L^{p'}(\Omega)}\leq C\|\nabla u\|_{L^{2}(\Omega)}\mathfrak m^{1/p'}\left(U\right).
\end{equation}
To prove \eqref{e202}, we distinguish two cases:
\begin{itemize}
\item[(i)]
The case $1<p\leq 2.$  In this case, 
\[1\leq r<2,\quad r:= \frac{2p'}{p'+2}.\]
Hence by Sobolev embedding  and H\"older's inequality we have that
 \begin{equation}\label{wzz1}
  \|u\|_{L^{p'}(\Omega)}\leq C \|\nabla u\|_{L^{r}(\Omega)}\leq C\|\bm 1_{U}\|_{L^{p'}(\Omega)} \|\nabla u\|_{L^{2}(\Omega)}=C \|\nabla u\|_{L^{2}(\Omega)}\mathfrak m^{1/p'}(U).
 \end{equation} 
 Note that in the first inequality of \eqref{wzz1} we have used  the Sobolev embedding $W^{1,r}_{0}(\Omega)\hookrightarrow L^{p'}(\Omega),$  and in the second inequality we have used H\"older's inequality and the fact that $u\geq 0$ in $\Omega$.  Hence \eqref{e202} has been proved.
 
 \item[(ii)] The case $2<p<+\infty.$ In this case,  
 \[2<s<+\infty,\quad s:=\frac{2p'}{2-p'}.\]
By H\"older's inequality we have  that
 \begin{equation}\label{e204}
 \|u\|_{L^{p'}(\Omega)}\leq \|u\|_{L^{2}(\Omega)}\|\bm 1_{U}\|_{L^{s}(\Omega)}=\|u\|_{L^{2}(\Omega)}\mathfrak m^{1/s}(U). 
 \end{equation}
 In view of the Sobolev embedding $W^{1,1}_{0}(\Omega)\hookrightarrow L^{2}(\Omega),$  $\|u\|_{L^{2}(\Omega)}$ can be estimated as follows:
 \begin{equation}\label{e205}
 \|u\|_{L^{2}(\Omega)} \leq  C\|\nabla u\|_{L^{1}(\Omega)} \leq C\|\nabla u\|_{L^{2}(\Omega)}\|\bm 1_{U}\|_{L^{2}(\Omega)}=C\|\nabla u\|_{L^{2}(\Omega)}\mathfrak m^{1/2}(U). 
 \end{equation}
The desired estimate \eqref{e202} follows from \eqref{e204} and \eqref{e205} immediately.
\end{itemize}
\end{proof}

For any Lebesgue measurable function $u:\mathbb R^2\to\mathbb R$, we use $u^*$ to denote its symmetric-decreasing rearrangement with respect to the origin. See \cite{LL}, \S 3.3 for the precise definition.  The following two rearrangement inequalities will be frequently used in later proofs.

\begin{lemma}[\cite{LL}, \S 3.4]\label{rri1}
Let $u,v$ be nonnegative Lebesgue measurable functions on $\mathbb R^2$. Then
\[\int_{\mathbb R^2}uvdx\leq \int_{\mathbb R^2}u^*v^*dx.\]
\end{lemma}

\begin{lemma}[\cite{LL}, \S 3.7]\label{rri2}
Let $u,v,w$ be nonnegative Lebesgue measurable functions on $\mathbb R^2$. Then
\[\int_{\mathbb R^2}\int_{\mathbb R^2}u(x)v(x-y)w(y)dxdy\leq \int_{\mathbb R^2}\int_{\mathbb R^2}u^*(x)v^*(x-y)w^*(y)dxdy.\]
\end{lemma}

The following lemma is a special case of Lemma 3.2 in \cite{BGu}, and will be used in the proof of Theorem \ref{thm2}.

\begin{lemma}[\cite{BGu}, Lemma 3.2]\label{bgu}
Let $1<p<+\infty$. Suppose $\{u_n\}\subset L^p(\mathbb R^2)$ satisfies for each $n$,
\begin{equation}\label{lily}
 u_n(\bm x)\geq 0 \,\, \mbox{ \rm a.e. }\bm x\in\mathbb R^2,\quad\int_{\mathbb R^2}\bm x u_n(\bm x)d\bm x=\mathbf 0,\quad {\rm supp}(u_n)\subset B_{r}(\bm 0) 
 \end{equation}
for some $r>0$. If $u_n\rightharpoonup u$ and $u^*_n\rightharpoonup v$ for some $u,v\in L^p(\mathbb R^2),$ then 
\[\int_{\mathbb R^2}\int_{\mathbb R^2}\ln\frac{1}{|\bm x-\bm y|}u(\bm x)u(\bm y)d\bm xd \bm y\leq \int_{\mathbb R^2}\int_{\mathbb R^2}\ln\frac{1}{|\bm  x-\bm  y|}v(\bm x)v(\bm y)d\bm  xd \bm y,\]
and  the equality implies $v=u=u^*$.
\end{lemma}

\section{Proof of Theorem \ref{thm1}}
 
 In this section, we give the proof of Theorem \ref{thm1}. The proof is based on the energy method established by Turkington \cite{T12}, which consists of three steps: basic energy estimates, suitable bound for the Lagrangian multipliers,   size and limiting location of the vortex cores.

For convenience, throughout this section we denote
\begin{equation}\label{dfka1}
\kappa_{\bm\varepsilon,1}=\int_{D}v^{+}d\bm x,\quad \kappa_{\bm\varepsilon,2}=-\int_{D}v^{-}d\bm x,\quad v\in\mathcal R_{\bm\varepsilon}.
\end{equation}
In view of \eqref{ep1}, we have that
\begin{equation}
\lim_{|\bm\varepsilon|\to 0}\kappa_{\bm\varepsilon,i}=\kappa_{i}, \quad i=1,2.
\end{equation}
Fix $\bar\varepsilon>0$  small enough  such that there exist two points $\bm x_{1}, \bm x_{2}\in D$ satisfying
\begin{equation}\label{e110}
 {B_{2\bar\varepsilon}(\bm x_{1})}\subset D, \quad  {B_{2\bar\varepsilon}(\bm x_{2})}\subset D,\quad {B_{2\bar\varepsilon}(\bm x_{1})}\cap  {B_{2\bar\varepsilon}(\bm x_{2})}=\varnothing.
\end{equation}
Without loss of generality, we assume that for any $|\bm\varepsilon|<\bar\varepsilon$,
\begin{equation}\label{wlgg1}
\frac{1}{2}|\kappa_{i}|\leq |\kappa_{\bm\varepsilon,i}|\leq\frac{3}{2}|\kappa_{i}|,\quad i=1,2,
\end{equation}
\begin{equation}\label{wlgg2}
 \varepsilon_{1}^{2/p'}\|v^{+}\|_{L^{p}(D)}\leq K_{1},\quad  \varepsilon_{2}^{2/p'}\|v^{-}\|_{L^{p}(D)}\leq K_{2},\quad\forall\,v\in\mathcal R_{\bm\varepsilon},\quad i=1,2,
\end{equation}
where $K_{1}, K_{2}>0$ do not depend on $\bm\varepsilon$. Note that \eqref{wlgg2} is doable by  \eqref{ep2}.  Combining \eqref{etc3} and \eqref{wlgg2}, we have that 
\begin{equation}\label{ljj1}
  \|\varrho_{\bm\varepsilon,1}\|_{L^{p}(\mathbb R^2)}\leq K_1,\quad  \|\varrho_{\bm\varepsilon,2}\|_{L^{p}(\mathbb R^2)}\leq K_2,\quad\forall\,|\bm\varepsilon|<\bar\varepsilon.
\end{equation}

In the rest of this paper, we always assume that $|\bm\varepsilon|<\bar\varepsilon$. For convenience, we shall denote by $C$ various positive constants possibly depending on $D,p,\kappa_{1},\kappa_{2}, K_{1}, K_{2}$ and $\bar\varepsilon$, but not on $\bm\varepsilon,$ whose values may change from line to line.

 \subsection{Basic energy estimates}

The purpose of this subsection is to prove the following energy estimates.
\begin{proposition}\label{p33}
There exists $C>0$ such that  for any  $\zeta\in\mathcal M_{\bm\varepsilon}$,
\begin{equation}\label{u301}
-\frac{\kappa^{2}_{\bm\varepsilon,1}}{4\pi}\ln\varepsilon_{1}-C\leq E(\zeta^{+})\leq -\frac{\kappa^{2}_{\bm\varepsilon,1}}{4\pi}\ln\varepsilon_{1}+C,
\end{equation}
\begin{equation}\label{u302}
-\frac{\kappa^{2}_{\bm\varepsilon,2}}{4\pi}\ln\varepsilon_{2}-C\leq E(\zeta^{-})\leq -\frac{\kappa^{2}_{\bm\varepsilon,2}}{4\pi}\ln\varepsilon_{2}+C,
\end{equation}
\begin{equation}\label{u303}
  \int_{D}\zeta^{+}\mathcal G\zeta^{-} d\mathbf x\leq  C.
\end{equation}
\end{proposition}

\begin{proof}
For clarity, the proof is divided into three steps:

\noindent{\bf Step 1.}  There exists $C>0$ such that 
\begin{equation}\label{u300}
E(\zeta)\geq -\frac{\kappa^{2}_{\bm\varepsilon,1}}{4\pi}\ln\varepsilon_{1}-\frac{\kappa^{2}_{\bm\varepsilon,2}}{4\pi}\ln\varepsilon_{2}-C\quad \forall\,\zeta\in\mathcal M_{\bm\varepsilon}.
\end{equation}
To prove \eqref{u300}, the idea is to choose a suitable test function and compute its energy.  Fix two points $\bm x_{1}, \bm x_{2}\in D$ such that \eqref{e110} holds.
As in Section 2,  let $\rho_{\bm\varepsilon,1}, \rho_{\bm\varepsilon,2}$ be the symmetric-decreasing rearrangement of $v^+, v^-$  with respect to the origin.
Denote
\[ v_{1}(\bm x)=\rho_{\bm\varepsilon,1}(\bm x-\bm x_{1}),\quad v_{2}(\bm x)=\rho_{\bm\varepsilon, 2}(\bm x-\bm x_{2}).\]
Then up to a set of zero Lebesgue measure,
\[\{\bm x\in D\mid v_i(\bm x)>0\}=B_{\varepsilon_{i}}(\bm x_{i}),\quad i=1,2.\]
It is easy to check that $v:=v_{1}-v_{2}\in\mathcal R_{\bm\varepsilon}$ since $|\bm\varepsilon|<\bar\varepsilon$. Hence 
\begin{equation}\label{eztb}
E(\zeta)\geq E(v) \quad\forall\,\zeta\in\mathcal M_{\bm\varepsilon}.
\end{equation} 
Now we estimate $E(v)$. Write 
\begin{equation}\label{e3002}
E(v)=\frac{1}{4\pi}\int_{D}\int_{D}\ln\frac{1}{|\bm x-\bm y|}v(\bm x)v(\bm y)d\bm xd\bm y-\frac{1}{2}\int_{D}\int_{D}h(\bm x,\bm y)v(\bm x)v(\bm y)d\bm xd\bm y.
\end{equation}
For the first integral on the right-hand side of \eqref{e3002}, we have that
\begin{equation}\label{fio1}
\begin{split}
 &\int_{D}\int_{D}\ln\frac{1}{|\bm x-\bm y|}v(\bm x)v(\bm y)d\bm xd\bm y\\
 =&\int_{B_{\varepsilon_{1}}(\bm x_{1})}\int_{B_{\varepsilon_{1}}(\bm x_{1})}\ln\frac{1}{|\bm x-\bm y|}v_{1}(\bm x)v_{1}(\bm y)d\bm xd\bm y+\int_{B_{\varepsilon_{2}}(\bm x_{2})}\int_{B_{\varepsilon_{2}}(\bm x_{2})}\ln\frac{1}{|\bm x-\bm y|}v_{2}(\bm x)v_{2}(\bm y)d\bm xd\bm y\\
 &-2\int_{B_{\varepsilon_{1}}(\bm x_{1})}\int_{B_{\varepsilon_{2}}(\bm x_{2})}\ln\frac{1}{|\bm x-\bm y|}v_{1}(\bm x)v_{2}(\bm y)d\bm xd\bm y\\
 \geq&-\kappa_{\bm\varepsilon,1}^{2}\ln(2\varepsilon_{1})-\kappa_{\bm\varepsilon,2}^{2}\ln(2\varepsilon_{2}) -2\kappa_{\bm\varepsilon,1}\kappa_{\bm\varepsilon,2}\ln(2\bar\varepsilon)\\
 \geq &-\kappa_{\bm\varepsilon,1}^{2}\ln \varepsilon_{1}-\kappa_{\bm\varepsilon,2}^{2}\ln\varepsilon_{2} -C.
    \end{split}
    \end{equation}
Note that in the second to last inequality we have used the following implications of \eqref{e110}:
 \[|\bm x-\bm y|\leq 2\varepsilon_{1}\quad\forall\,\bm x,\bm y\in B_{\varepsilon_{1}}(\bm x_{1}),\]
  \[|\bm x-\bm y|\leq 2\varepsilon_{2}\quad\forall\,\bm x,\bm y\in B_{\varepsilon_{2}}(\bm x_{2}),\]
 \[|\bm x-\bm y|\geq 2\bar\varepsilon\quad\forall\,\bm x\in B_{\varepsilon_{1}}(\bm x_{1}), \bm y\in B_{\varepsilon_{2}}(\bm x_{2}).\]
For the second integral on the right-hand side of \eqref{e3002}, we have that
\begin{equation}\label{fio2}
\begin{split}
& \int_{D}\int_{D}h(\bm x,\bm y)v(\bm x)v(\bm y)d\bm xd\bm y \\
=&\int_{B_{\varepsilon_{1}}(\bm x_{1})}\int_{B_{\varepsilon_{1}}(\bm x_{1})}h(\bm x,\bm y)v_{1}(\bm x)v_{1}(\bm y)d\bm xd\bm y+\int_{B_{\varepsilon_{2}}(\bm x_{2})}\int_{B_{\varepsilon_{2}}(\bm x_{2})}h(\bm x,\bm y)v_{2}(\bm x)v_{2}(\bm y)d\bm xd\bm y\\
 &-2\int_{B_{\varepsilon_{1}}(\bm x_{1})}\int_{B_{\varepsilon_{2}}(\bm x_{2})}h(\bm x,\bm y)v_{1}(\bm x)v_{2}(\bm y)d\bm xd\bm y\\
 \leq&\kappa_{\bm\varepsilon,1}^{2} \|h\|_{L^{\infty}(B_{\bar\varepsilon}(\bm x_{1})\times B_{\bar\varepsilon}(\bm x_{1}))}+\kappa_{\bm\varepsilon,2}^{2} \|h\|_{L^{\infty}(B_{\bar\varepsilon}(\bm x_{2})\times B_{\bar\varepsilon}(\bm x_{2}))}-2\kappa_{\bm\varepsilon,1}\kappa_{\bm\varepsilon,2} \|h\|_{L^{\infty}(B_{\bar\varepsilon}(\bm x_{1})\times B_{\bar\varepsilon}(\bm x_{2}))}\\
\leq& C.
    \end{split}
    \end{equation}
Here we used the fact that $h$ is smooth (thus locally bounded) in $D\times D$. The desired estimate \eqref{u300} follows from \eqref{fio1} and \eqref{fio2}.

\noindent{\bf Step 2.}  There exists $C>0$ such that  
\begin{equation}\label{e3003}
E(\zeta^{+})\leq -\frac{\kappa^{2}_{\bm\varepsilon,1}}{4\pi}\ln\varepsilon_{1}+C, 
\quad 
E(\zeta^{-})\leq -\frac{\kappa^{2}_{\bm\varepsilon,2}}{4\pi}\ln\varepsilon_{2}+C,\quad\forall\, \zeta\in\mathcal M_{\bm\varepsilon}.
\end{equation}
We only prove the estimate for $E(\zeta^+)$. The estimate for $E(\zeta^-)$  can be obtained analogously. For any $\zeta\in\mathcal M_{\bm\varepsilon}$, write
\begin{equation}\label{kpp1}
E(\zeta^{+})=\frac{1}{4\pi}\int_{D}\int_{D}\ln\frac{1}{|\bm x-\bm y|}\zeta^{+}(\bm x)\zeta^{+}(\bm y)d\bm xd\bm y-\frac{1}{2}\int_{D}\int_{D}h(\bm x,\bm y)\zeta^{+}(\bm x)\zeta^{+}(\bm y)d\bm xd\bm y.
\end{equation}
The first integral in \eqref{kpp1} can be estimated 
as follows:
\begin{align*}
&\int_{D}\int_{D}\ln\frac{1}{|\bm x-\bm y|}\zeta^{+}(\bm x)\zeta^{+}(\bm y)d\bm xd\bm y\\
\leq& \int_{\mathbb R^{2}}\int_{\mathbb R^{2}}\ln\frac{1}{|\bm x-\bm y|}\rho_{\bm\varepsilon,1}(\bm x)\rho_{\bm\varepsilon,1}(\bm y)d\bm xd\bm y\\
=&\frac{1}{\varepsilon_{1}^{4}}\int_{\mathbb R^{2}}\int_{\mathbb R^{2}}\ln\frac{1}{|\bm x-\bm y|}\varrho_{\bm\varepsilon,1}\left(\frac{\bm x}{\varepsilon_{1}}\right)\varrho_{\bm\varepsilon,1}\left(\frac{\bm y}{\varepsilon_{1}}\right)d\bm xd\bm y\\
=&\int_{B_{1}(\bm 0)}\int_{B_{1}(\bm 0)}\ln\frac{1}{|\varepsilon_{1}\bm x-\varepsilon_{1}\bm y|}\varrho_{\bm\varepsilon,1}(\bm x)\varrho_{\bm\varepsilon,1}(\bm y)d\bm xd\bm y\\
=&-\kappa_{\bm\varepsilon,1}^{2}\ln\varepsilon_{1}+\int_{B_{1}(\bm 0)}\int_{B_{1}(\bm 0)}\ln\frac{1}{|\bm x-\bm y|}\varrho_{\bm\varepsilon,1}(\bm x)\varrho_{\bm\varepsilon,1}(\bm y)d\bm xd\bm y\\
\leq &-\kappa_{\bm\varepsilon,1}^{2}\ln\varepsilon_{1}+\kappa_{\bm\varepsilon,1}\sup_{\bm x\in B_1(\bm 0)}\int_{B_{1}(\bm 0)}\ln \frac{1}{|\bm x-\bm y|} \varrho_{\bm\varepsilon,1}(\bm y)d\bm y \\
\leq &-\kappa_{\bm\varepsilon,1}^{2}\ln\varepsilon_{1}+\kappa_{\bm\varepsilon,1} \int_{B_{1}(\bm 0)}\ln \frac{1}{|\bm y|} \varrho_{\bm\varepsilon,1}(\bm y)d\bm y\\
\leq &-\kappa_{\bm\varepsilon,1}^{2}\ln\varepsilon_{1}+\kappa_{\bm\varepsilon,1}\left\| \ln |\bm y| \right\|_{L^{p'}(B_{1}(\bm 0))}\left\| \varrho_{\bm\varepsilon,1} \right\|_{L^{p}(\mathbb R^{2})}\\
\leq &-\kappa_{\bm\varepsilon,1}^{2}\ln\varepsilon_{1}+C.
\end{align*}
Note that we have used  Lemma \ref{rri2} in the first inequality,  Lemma \ref{rri1} in the third to last inequality, and \eqref{ljj1} in the last inequality. For the second integral in \eqref{kpp1}, since $h$ is bounded from below in $D\times D,$ we have that
\begin{equation*}
\int_{D}h(\bm x,\bm y)\zeta^{+}(\bm x)\zeta^{+}(\bm y)d\bm xd\bm y\geq \kappa^{2}_{\bm\varepsilon,1}\inf_{\bm x,\bm y\in D }h(\bm x,\bm y)\geq -C.
\end{equation*}
This completes the proof.

\noindent{\bf Step 3.}  Now we are ready to prove \eqref{u301}-\eqref{u303} based on Step 1 and Step 2.
First notice that
\begin{equation}\label{o30}
E(\zeta)=E(\zeta^{+})+E(\zeta^{-})-\int_{D}\zeta^{+}\mathcal G\zeta^{-}d\bm x.
\end{equation}
Since $G(\bm x,\bm y)>0$ for any $\bm x, \bm y\in D, \bm x\neq \bm y$, we deduce that
\begin{equation}\label{o31}
\int_{D}\zeta^{+}\mathcal G\zeta^{-}d\bm x> 0.
\end{equation}
Combining \eqref{u300}, \eqref{o30} and \eqref{o31}, we obtain
\begin{equation}\label{o32}
E(\zeta^{+})+E(\zeta^{-})\geq -\frac{\kappa^{2}_{\bm\varepsilon,1}}{4\pi}\ln\varepsilon_{1}-\frac{\kappa^{2}_{\bm\varepsilon,2}}{4\pi}\ln\varepsilon_{2}-C.
\end{equation}
From \eqref{e3003} and \eqref{o32}, we  obtain
\[E(\zeta^{+}) \geq -\frac{\kappa^{2}_{\bm\varepsilon,1}}{4\pi}\ln\varepsilon_{1} -C,\quad E(\zeta^{-}) \geq -\frac{\kappa^{2}_{\bm\varepsilon,2}}{4\pi}\ln\varepsilon_{2} -C.\]
Finally \eqref{u303} follows from \eqref{u300}, \eqref{e3003} and \eqref{o30}. 

\end{proof}

\subsection{Estimates for Lagrangian multipliers}
By Theorem \ref{thm0}, for any
$\zeta\in\mathcal M_{\bm\varepsilon}$ there exists  some  increasing function  $\phi_\zeta:\mathbb R\to\mathbb R\cup\{\pm\infty\}$ such that
\begin{equation}\label{fv1}
\zeta=\phi_\zeta(\mathcal G\zeta) \quad\mbox{a.e. in $D$, }
\end{equation}
 Based on this fact, we can define the Lagrangian multipliers $\mu_{\zeta,1}, \mu_{\zeta,2}$ related to $\zeta$ as follows:
 \begin{equation}\label{lmp}
 \mu_{\zeta,1}=\inf\{s\in\mathbb R\mid \phi_\zeta(s)>0\},\quad  \mu_{\zeta,2}=\sup\{s\in\mathbb R\mid \phi_\zeta(s)<0\}.
 \end{equation}

 \begin{lemma}\label{vcc1}
 For any $\zeta\in\mathcal M_{\bm\varepsilon},$ it holds that 
 \begin{equation}\label{vc90}
 V^+_\zeta=\{\bm x\in D\mid \mathcal G\zeta(\bm x)>\mu_{\zeta,1}\},
 \end{equation}
 \begin{equation}\label{vc91}
 V^-_\zeta=\{\bm x\in D\mid\mathcal G\zeta(\bm x)<\mu_{\zeta,2}\},
 \end{equation}
 where $ V^\pm_\zeta$ is defined by \eqref{defvc}.
 \end{lemma}
 \begin{proof}
 We only prove \eqref{vc90}.
 From the definition of $\mu_{\zeta,1}$,  it is clear that 
\begin{equation}\label{jabi1}
\zeta>0\quad \mbox{ a.e. in }\{\bm x\in D \mid  \mathcal G\zeta(\bm x)> \mu_{\zeta,1}\},
\end{equation}
\begin{equation}\label{jabi2}
\zeta=0  \quad  \mbox{ a.e. in }\{\bm x\in D \mid  \mathcal G\zeta(\bm x)< \mu_{\zeta,1}\}.
\end{equation}
On the level set $\{\bm x\in D\mid  \mathcal G\zeta(\bm x)= \mu_{\zeta,1}\},$ using the fact that  all  weak derivatives of a Sobolev function vanish on its level sets (see \cite{EV},  p. 153), we conclude that
\begin{equation}\label{jabi3}
\zeta=- \Delta \mathcal G\zeta=0\quad \mbox{ a.e. on }\{\bm x\in D \mid  \mathcal G\zeta(\bm x)=\mu_{\zeta,1}\}.
\end{equation} 
From \eqref{jabi1}-\eqref{jabi3} we get the desired result.
 \end{proof}

Now we turn to the estimates for $\mu_{\zeta,1}$ and $\mu_{\zeta,2}$. 
 
\begin{lemma}\label{bdd0}
There exists $C>0$ such that 
\begin{equation}\label{mosy1}
\mu_{\zeta,1}\geq -C,\quad \mu_{\zeta,2}\leq C,\quad \forall\, \zeta\in\mathcal M_{\bm\varepsilon}.
\end{equation}
\end{lemma}
\begin{proof}
We argue by contradiction. Without loss of generality, we assume that $\mu_{\zeta,2}$ is not bounded from above as $|\bm\varepsilon|\to0$, i.e.,  there exist $\{\bm\varepsilon_n\}$ with $|\bm\varepsilon_n|\to0$ as $n\to+\infty$ and $\zeta_n\in \mathcal M_{\bm\varepsilon_{n}}$ for each $n$, such that 
\begin{equation}\label{mosy2}
\mu_{n,2}:=\mu_{\zeta_{n},2}\to+\infty\quad\mbox{as }n\to+\infty.
\end{equation}
Since $\mu_{n,1}:=\mu_{\zeta_{n},1}\geq \mu_{n,2},$  it follows that
$\mu_{n,1}\to+\infty$ as $n\to+\infty.$

We claim that there exists $C>0$ such that
\begin{equation}\label{e310}
\mu_{n,1}\geq -\frac{\kappa_{n,1}}{2\pi}\ln\varepsilon_{n,1}-C 
\end{equation}
for sufficiently large $n$. 
First by Proposition \ref{p33} we have that 
\begin{equation}\label{e308}
\int_{D}\zeta_{n}^{+}\mathcal G\zeta_{n}d\bm x\geq -\frac{\kappa^{2}_{n,1}}{2\pi}\ln\varepsilon_{n,1}-C,
\end{equation}
where $\kappa_{n,1}:=\kappa_{\bm\varepsilon_{n},1},$ $\bm\varepsilon_{n}=(\varepsilon_{n,1},\varepsilon_{n,2}).$
On the other hand,  for  sufficiently large $n$ such that $\mu_{n,1}>0,$  
we have that $u_n:=(\mathcal G\zeta_{n}-\mu_{n,1})^{+}\in W^{2,p}\cap W^{1,p}_{0}(D).$  Hence by integration by parts,
\begin{equation}\label{e309}
\begin{split}
\int_{D}\zeta_{n}^{+}\mathcal G\zeta_{n}d\bm x&=\int_{D}\zeta_{n}^{+}(\mathcal G\zeta_{n}-\mu_{n,1})d\bm x+\kappa_{n,1}\mu_{n,1}\\
&=\int_{D}\zeta_{n}^{+}u_nd\bm x+\kappa_{n,1}\mu_{n,1}\\
&=\int_D|\nabla u_n|^2dx+\kappa_{n,1}\mu_{n,1}.
 \end{split}
 \end{equation}
 Note that in the second equality we used \eqref{vc90}. Applying Lemma \ref{le21} and recalling \eqref{wlgg2}, we have that
\begin{equation}\label{z300}
\begin{split}
 \int_D|\nabla u_n|^2dx&\leq C\|\Delta u_n\|_{L^p(D)}\mathfrak m^{1/p'}(\{\bm x\in D\mid u_n(\bm x)>0\})\\
 &=C\|\zeta^+_n\|_{L^p(D)}\mathfrak m^{1/p'}(\{\bm x\in D\mid \zeta^+_n(\bm x)>0\})\\
 &=C\varepsilon^{2/p'}_{n,1}\|\zeta^+_n\|_{L^p(D)}\\
 &\leq C.
 \end{split}
 \end{equation}
 Combining \eqref{e308}-\eqref{z300}, we obtain \eqref{e310} .

 We proceed to deduce a contradiction.
By \eqref{vc90},  it holds that 
\[\mathcal G\zeta_{n}(\bm x)>\mu_{n,1} \quad \forall\,\bm x\in V_n^+:=V_{\zeta_n}^+.\]
Taking into account \eqref{e310}, we have that
 \begin{equation}\label{gzta1}
 \mathcal G\zeta_{n}(\bm x)\geq -\frac{\kappa_{n,1}}{2\pi}\ln\varepsilon_{n,1}-C  \quad \forall\,\bm x\in V_n^+.
 \end{equation}
 Write
 \begin{equation}\label{gzta2}
 \mathcal G\zeta_n(\bm x)=\int_{D}-\frac{1}{2\pi}\ln|\bm x-\bm y|\zeta_{n}^{+}(\bm y)d\bm y-\int_{D}h(\bm x,\bm y)\zeta_{n}^{+}(\bm y)d\bm y-\mathcal G\zeta^-(\bm x),\quad \bm x\in D.
 \end{equation}
Since  $h$ is bounded from below in $D\times D$ and $\mathcal G\zeta_n^-$ is nonnegative in $D$, we obtain from \eqref{gzta2} that 
 \begin{equation}\label{gzta3}
\int_{D}-\frac{1}{2\pi}\ln|\bm x-\bm y|\zeta_{n}^{+}(\bm y)d\bm y\geq  \mathcal G\zeta_n(\bm x)-C\quad\forall\, \bm x\in D. 
 \end{equation}
From \eqref{gzta1} and \eqref{gzta3}, we obtain that 
 \begin{equation}\label{e312}
\int_{D} \ln\frac{\varepsilon_{n,1}}{|\bm x-\bm y|}\zeta_{n}^{+}(\bm y)d\bm y\geq  -C \quad \forall\,\bm x\in V_n^+.
\end{equation}
Choose $L>1$ to be determined later, we have that
 \begin{equation}\label{e313}
\int_{B_{L\varepsilon_{n,1}}(\bm x)} \ln\frac{\varepsilon_{n,1}}{|\bm x-\bm y|}\zeta_{n}^{+}(\bm y)d\bm y+\int_{D\setminus B_{L\varepsilon_{n,1}}(\bm x)} \ln\frac{\varepsilon_{n,1}}{|\bm x-\bm y|}\zeta_{n}^{+}(\bm y)d\bm y\geq  -C\quad \forall\,\bm x\in V_n^+.
\end{equation} 
For the first term on the left-hand side of \eqref{e313}, we have for any $\bm x$ that
 \begin{equation}\label{e314}
 \begin{split}
\int_{B_{L\varepsilon_{n,1}}(\bm x)} \ln\frac{\varepsilon_{n,1}}{|\bm x-\bm y|}\zeta_{n}^{+}(\bm y)d\bm y &\leq \int_{B_{L\varepsilon_{n,1}}(\bm 0)} \ln\frac{\varepsilon_{n,1}}{|\bm y|}\rho_{n,1}(\bm y)d\bm y\quad(\rho_{n,1}:=\rho_{\bm\varepsilon_{n},1})\\
&=\int_{B_{\varepsilon_{n,1}}(\bm 0)} \ln\frac{\varepsilon_{n,1}}{|\bm y|}\rho_{n,1}(\bm y)d\bm y\\
&\leq \|\rho_{n,1}\|_{L^{p}(B_{\varepsilon_{n,1}}(\bm 0))}\left\|\ln\frac{\varepsilon_{n,1}}{|\bm y|}\right\|_{L^{p'}(B_{\varepsilon_{n,1}}(\bm 0))}\\
&\leq \varepsilon_{n,1}^{p'/2}\|\zeta_{n}^{+}\|_{L^{p}(D)}\|\ln |\bm y|\|_{L^{p'}(B_{1}(\bm 0))}\\
&\leq C.
\end{split}
\end{equation}
Note that in the first inequality of \eqref{e314} we used Lemma \ref{rri1}, and in the last inequality we used \eqref{wlgg2}.
For the second integral in \eqref{e313}, we have for any $\bm x$ that
 \begin{equation}\label{e315}
\int_{D\setminus B_{L\varepsilon_{n,1}}(\bm x)} \ln\frac{\varepsilon_{n,1}}{|\bm x-\bm y|}\zeta_{n}^{+}(\bm y)d\bm y\leq  -\ln L\int_{D\setminus B_{L\varepsilon_{n,1}}(\bm x)}  \zeta_{n}^{+}(\bm y)d\bm y.
\end{equation} 
From \eqref{e313}-\eqref{e315}, we obtain
 \begin{equation}\label{e316}
 \int_{D\setminus B_{L\varepsilon_{n,1}}(\bm x)}  \zeta_{n}^{+}(\bm y)d\bm y\leq \frac{C}{\ln L}\quad \forall\,\bm x\in V_n^+.
\end{equation} 
Choosing $L$ large enough (not depending on $n$) such that 
 \begin{equation}\label{e317}
 \int_{D\setminus B_{L\varepsilon_{n,1}}(\bm x)}  \zeta_{n}^{+}(\bm y)d\bm y\leq \frac{1}{3}\kappa_{n,1},
\end{equation} 
we obtain from \eqref{e316} that
 \begin{equation}\label{e318}
 \int_{ B_{L\varepsilon_{n,1}}(\bm x)}  \zeta_{n}^{+}(\bm y)d\bm y\geq \frac{2}{3}\kappa_{n,1}\quad \forall\,\bm x\in V_n^+.
\end{equation} 
With \eqref{e318} we can argue by contradiction that
 \begin{equation}\label{e319}
 \mbox{\rm diam(}V_n^+)\leq 2L\varepsilon_{n,1},
\end{equation} 
provided that $n$ is large enough.

Now we are ready to deduce a contradiction. Fix some small positive number $\delta$ such that 
\begin{equation}\label{nlss1}
\pi\delta^{2}<\frac{1}{2}\mathfrak m(D).
\end{equation}
From \eqref{e319}, we have for large $n$ that
\[ V^+_n\subset B_{\delta/2}(\bm z_{n}),\quad \bm z_{n}:=\frac{1}{\kappa_{n,1}}\int_{D}\zeta_n^{+}d\bm x.\]
Then for any $\bm x\in D\setminus B_{\delta}(\bm z_{n}),$  
\begin{equation}\label{hko1}
\begin{split}
\mathcal G\zeta_n(\bm x)&=\mathcal G\zeta_n^+(\bm x)-\mathcal G\zeta_n^-(\bm x)\\
&\leq  \mathcal G\zeta_n^+(\bm x)\\
&=\int_D-\frac{1}{2\pi}\ln|\bm x-\bm y|\zeta^+_n(\bm y)d\bm y-\int_Dh(\bm x,\bm y)\zeta^+_n(\bm y)d\bm y\\
&=\int_{B_{\delta/2}(\bm z_n)}-\frac{1}{2\pi}\ln|\bm x-\bm y|\zeta^+_n(\bm y)d\bm y-\int_Dh(\bm x,\bm y)\zeta^+_n(\bm y)d\bm y\\
&\leq  -\frac{1}{2\pi}\kappa_{n,1}\ln\left(\frac{\delta}{2}\right)-\kappa_{n,1}\inf_{\bm x,\bm y\in D}h(\bm x,\bm y)\\
&\leq C.
\end{split}
\end{equation}
Combining \eqref{hko1} and the fact that $\mu_{n,2}\to+\infty$ as $n\to+\infty$, we conclude for large $n$  that 
\begin{equation}\label{nlss2}
D\setminus B_{\delta}(\bm z_{n})\subset \{\bm x\in D\mid \mathcal G\zeta_{n}(\bm x)<\mu_{n,2}\}=\{\bm x\in D\mid \zeta_{n}^{-}(\bm x)>0\}.
\end{equation}
Now we can easily deduce a contradiction: by \eqref{ep0} it holds that  
\[\mathfrak m(\{\bm x\in D\mid \zeta_{n}^{-}(\bm x)>0\})=\pi\varepsilon^2_{n,2}\to0\]
as $n\to+\infty$, while by \eqref{nlss1} and \eqref{nlss2},
\[\mathfrak m(\{\bm x\in D\mid \zeta_{n}^{-}(\bm x)>0\})\geq \mathfrak m(D\setminus B_{\delta}(\bm z_{n}))\geq \mathfrak m(D)-\pi\delta^{2}>\frac{1}{2}\mathfrak m(D) \]
for every sufficiently large $n$.
\end{proof}

Based on Lemma \ref{bdd0},  we can prove  better a estimate  for $\mu_{\zeta,1}$ or $\mu_{\zeta,2}$.
\begin{proposition}\label{bdd00}
There exists $C>0$ such that 
\begin{equation}\label{b3000}
\mu_{\zeta,1}\geq -\frac{\kappa_{\bm\varepsilon,1}}{2\pi}\ln\varepsilon_{1}-C,
\quad 
\mu_{\zeta,2}\leq  \frac{\kappa_{\bm\varepsilon,2}}{2\pi}\ln\varepsilon_{2}+C,\quad  \forall\,\zeta\in\mathcal M_{\bm\varepsilon}.
\end{equation}
\end{proposition}
\begin{proof}
We only prove the estimate for $\mu_{\zeta,1}$.  By Lemma \ref{bdd0}, we can choose $C_{1}>0$, not depending on $\bm\varepsilon,$ such that $\mu_{\zeta,1}\geq -C_{1}$. 
By Proposition \ref{p33},
\begin{equation}\label{b300}
\int_{D}\zeta^{+}\mathcal G\zeta d\bm x\geq -\frac{\kappa^{2}_{\bm\varepsilon,1}}{2\pi}\ln\varepsilon_{1}-C\quad  \forall\,\zeta\in\mathcal M_{\bm\varepsilon}.
\end{equation}
which can also be written as follows:
\begin{equation}\label{b301}
\int_{D}\zeta^{+}(\mathcal G\zeta-\mu_{\zeta,1}-C_{1}) d\bm x+(\mu_{\zeta,1}+C_{1})\kappa_{\bm\varepsilon,1}\geq -\frac{\kappa^{2}_{\bm\varepsilon,1}}{2\pi}\ln\varepsilon_{1}-C.
\end{equation}
Denote $u:=(\mathcal G\zeta-\mu_{\zeta,1}-C_{1})^{+}$. Then  we obtain from   \eqref{b301} that
\begin{equation}\label{b302}
\int_{D}\zeta^{+}u d\bm x+\mu_{\zeta,1} \kappa_{\bm\varepsilon,1}\geq -\frac{\kappa^{2}_{\bm\varepsilon,1}}{2\pi}\ln\varepsilon_{1}-C\quad  \forall\,\zeta\in\mathcal M_{\bm\varepsilon}.
\end{equation}
By the choice of $C_{1}$, we see that $u\in W^{1,p}_{0}(D).$ Moreover, it is easy to check that 
\[
-\Delta u= \begin{cases}
\zeta^{+} &\mbox{\rm a.e. in } \{\bm x\in D\mid u(\bm x)>0\},\\
 0&\mbox{\rm a.e. in } \{\bm x\in D\mid u(\bm x)=0\}.
 \end{cases}
\]
Hence we can apply integration by parts to estimate the integral in \eqref{b302} as follows:
\begin{equation}\label{b304}
\begin{split}
\int_{D}\zeta^{+}u d\bm x&=\int_{D}|\nabla u|^{2} d\bm x\\
 &\leq C\|\Delta u\|_{L^{p}(D)}\mathfrak m^{1/p'}(\{\bm x\in D\mid u(\bm x)>0\})\\
 &\leq  C\|\zeta^{+}\|_{L^{p}(D)}\mathfrak m^{1/p'}(\{\bm x\in D\mid \zeta(\bm x)>0\})\\
 &=C\varepsilon_1^{2/p'}\|\zeta^{+}\|_{L^{p}(D)} \\
 &\leq C.
 \end{split}
 \end{equation}
 Here we have used  Lemma \ref{le21} in the first inequality and \eqref{wlgg2} in the last inequality. The desired estimate for $\mu_{\zeta,1}$ follows from 
 \eqref{b302} and \eqref{b304} immediately.
 \end{proof}

\subsection{Size of vortex cores}
With the  estimates for the Lagrangian multipliers in the last subsection, we are ready to  estimate   the size of the vortex cores.
\begin{proposition}\label{prp45}
There exists $C>0$ such that
\begin{equation}\label{pp1}
\mbox{\rm diam(}V^+_\zeta)\leq C\varepsilon_{1},\quad
\mbox{\rm diam(}V^-_\zeta)\leq C\varepsilon_{2},\quad \forall\,\zeta\in\mathcal M_{\bm\varepsilon}.
\end{equation}
\end{proposition}

\begin{proof}
We only prove the estimate for diam$(V^+_\zeta)$. Fix $\zeta\in\mathcal M_{\bm\varepsilon}$ and $\bm x\in V_\zeta^+$. By \eqref{vc90} it holds that
$\mathcal G\zeta(\bm x)>\mu_{\zeta,1},$ 
and thus
\begin{equation}\label{faz0}
\mathcal G\zeta^{+}(\bm x)>\mu_{\zeta,1}.
\end{equation}
In view of \eqref{b3000}, we obtain
\begin{equation}\label{faz1}
\mathcal G\zeta^{+}(\bm x) \geq -\frac{\kappa_{\bm\varepsilon,1}}{2\pi}\ln\varepsilon_{1}-C,
\end{equation}
which can also be written as
\begin{equation}\label{faz2}
\int_{D}-\frac{1}{2\pi}\ln {|\bm x-\bm y|}\zeta^{+}(\bm y)d\bm y-\int_{D}h(\bm x,\bm y)\zeta^{+}(\bm y)d\bm y\geq  -\frac{\kappa_{\bm\varepsilon,1}}{2\pi}\ln\varepsilon_{1}-C. 
\end{equation}
Taking into account the fact that $h$ is bounded from below in $D\times D,$ we get from \eqref{faz2} that
\[
\frac{1}{2\pi}\int_{D}\ln\frac{1}{|\bm x-\bm y|}\zeta^{+}(\bm y)d\bm y\geq  -\frac{\kappa_{\bm\varepsilon,1}}{2\pi}\ln\varepsilon_{1}-C,
\]
or equivalently,
\begin{equation}\label{faz3}
\int_{D}\ln\frac{\varepsilon_{1}}{|\bm x-\bm y|}\zeta^{+}(\bm y)d\bm y\geq -C.
\end{equation}
Now we can repeat the argument from \eqref{e312} to \eqref{e319} to obtain the desired estimate for diam$(V^+_\zeta)$.

\end{proof}

\subsection{Limiting location  of vortex cores}
By Proposition \ref{prp45}, for any $\zeta\in\mathcal M_{\bm\varepsilon}$ the associated positive and negative vortex cores  ``shrink" to two points in $\bar D$ as $|\bm\varepsilon|$ vanishes (at least along some subsequence). 
In this subsection we study the  location of these two points.
\begin{proposition}\label{llvv}
Fix a sequence $\{\bm\varepsilon_{n}\}$ such that $|\bm\varepsilon_{n}|\to 0$ as $n\to+\infty$, and a sequence $\{\zeta_{n}\}$ such that $\zeta_{n}\in\mathcal M_{\bm\varepsilon_{n}}$ for each $n$. Suppose there exist two points $\bar{\bm x}_{1},\bar{\bm x}_{2}\in \bar D$ such that
\begin{equation}\label{pp2}
\mathbf X(\zeta^{+}_{n})\to\bar{\bm x}_{1},\quad \mathbf X(\zeta^{-}_{n})\to\bar{\bm x}_{2},
\end{equation}
where $\mathbf X(\zeta^{+}_{n})$ and $ \mathbf X(\zeta^{-}_{n})$ are defined as in \eqref{dece}.
Then $ \bar{\bm x}_{1},\bar{\bm x}_{2}\in  D$, $ \bar{\bm x}_{1}\neq \bar{\bm x}_{2}$, and  $(\bar{\bm x}_{1},\bar{\bm x}_{2})$ is a global minimum point of $W_{\bm \kappa}$ with $\bm\kappa=(\kappa_{1},\kappa_{2})$.

\end{proposition}
\begin{proof}
Fix $\bm x_{1},\bm x_{2}\in D,$ $\bm x_{1}\neq \bm x_{2}.$ 
 Define a sequence of test functions $\{v_n\},$
\begin{equation}\label{yyss}
v_{n}=v_{n,1}-v_{n,2},\quad v_{n,1}(\bm x):=\rho_{\bm\varepsilon_n,1}(\bm x-\bm x_{1}),\quad v_{n,2}(\bm x):=\rho_{\bm\varepsilon_n,2}(\bm x-\bm x_{2}),\end{equation}
where, as in Section 2, $\rho_{\bm\varepsilon_n,1}$ and $\rho_{\bm\varepsilon_n,2}$ are the symmetric-decreasing rearrangement of $\zeta_n^+$ and $\zeta_n^-$  with respect to the origin, respectively.
It is easy to see that    $v_{n}\in\mathcal R_{\bm\varepsilon_{n}}$ if $n$ is large enough. Hence we have that 
 \begin{equation}\label{bk1}
E(\zeta_{n})\geq E(v_{n})\quad\forall\,n.
\end{equation}
For $E(\zeta_{n}),$ we have that
\begin{equation}\label{bk2}
\begin{split}
 E(\zeta_{n})
=&\frac{1}{4\pi}\int_{D}\int_{D}\ln\frac{1}{|\bm x-\bm y|}\zeta^{+}_{n}(\bm x)\zeta^{+}_{n}(\bm y)d\bm xd\bm y+\frac{1}{4\pi}\int_{D}\int_{D}\ln\frac{1}{|\bm x-\bm y|}\zeta^{-}_{n}(\bm x)\zeta^{-}_{n}(\bm y)d\bm xd\bm y\\
&-\frac{1}{2\pi}\int_{D}\int_{D}\ln\frac{1}{|\bm x-\bm y|}\zeta^{+}_{n}(\bm x)\zeta^{-}_{n}(\bm y)d\bm xd\bm y-\frac{1}{2}\int_{D}\int_{D}h(\bm x,\bm y)\zeta^{+}_{n}(\bm x)\zeta^{+}_{n}(\bm y)d\bm xd\bm y\\
&-\frac{1}{2}\int_{D}\int_{D}h(\bm x,\bm y)\zeta^{-}_{n}(\bm x)\zeta^{-}_{n}(\bm y)d\bm xd\bm y+\int_{D}\int_{D}h(\bm x,\bm y)\zeta^{+}_{n}(\bm x)\zeta^{-}_{n}(\bm y)d\bm xd\bm y.
\end{split}
\end{equation}
For $E(v_{n}),$ we have that
\begin{equation}\label{bk3}
\begin{split}
E(v_{n})=&\frac{1}{4\pi}\int_{D}\int_{D}\ln\frac{1}{|\bm x-\bm y|}v_{n,1}(\bm x)v_{n,1}(\bm y)d\bm xd\bm y+\frac{1}{4\pi}\int_{D}\int_{D}\ln\frac{1}{|\bm x-\bm y|}v_{n,2}(\bm x)v_{n,2}(\bm y)d\bm xd\bm y\\
&-\frac{1}{2\pi}\int_{D}\int_{D}\ln\frac{1}{|\bm x-\bm y|}v_{n,1}(\bm x)v_{n,2}(\bm y)d\bm xd\bm y-\frac{1}{2}\int_{D}\int_{D}h(\bm x,\bm y)v_{n,1}(\bm x)v_{n,1}(\bm y)d\bm xd\bm y\\
&-\frac{1}{2}\int_{D}\int_{D}h(\bm x,\bm y)v_{n,2}(\bm x)v_{n,2}(\bm y)d\bm xd\bm y+\int_{D}\int_{D}h(\bm x,\bm y)v_{n,1}(\bm x)v_{n,2}(\bm y)d\bm xd\bm y.
\end{split}
\end{equation}
Applying Lemma \ref{rri2}, we have that
\begin{equation}\label{bk4}
 \int_{D}\int_{D}\ln\frac{1}{|\bm x-\bm y|}v_{n,1}(\bm x)v_{n,1}(\bm y)d\bm xd\bm y \geq \int_{D}\int_{D}\ln\frac{1}{|\bm x-\bm y|}\zeta^{+}_{n}(\bm x)\zeta^{+}_{n}(\bm y)d\bm xd\bm y,
\end{equation}
\begin{equation}\label{bk5}
 \int_{D}\int_{D}\ln\frac{1}{|\bm x-\bm y|}v_{n,2}(\bm x)v_{n,2}(\bm y)d\bm xd\bm y \geq \int_{D}\int_{D}\ln\frac{1}{|\bm x-\bm y|}\zeta^{-}_{n}(\bm x)\zeta^{-}_{n}(\bm y)d\bm xd\bm y.
\end{equation}
By \eqref{bk1}-\eqref{bk5}, we deduce that
\begin{equation}\label{bk6}
\begin{split}
&-\frac{1}{2\pi}\int_{D}\int_{D}\ln\frac{1}{|\bm x-\bm y|}\zeta^{+}_{n}(\bm x)\zeta^{-}_{n}(\bm y)d\bm xd\bm y-\frac{1}{2}\int_{D}\int_{D}h(\bm x,\bm y)\zeta^{+}_{n}(\bm x)\zeta^{+}_{n}(\bm y)d\bm xd\bm y\\
&-\frac{1}{2}\int_{D}\int_{D}h(\bm x,\bm y)\zeta^{-}_{n}(\bm x)\zeta^{-}_{n}(\bm y)d\bm xd\bm y+\int_{D}\int_{D}h(\bm x,\bm y)\zeta^{+}_{n}(\bm x)\zeta^{-}_{n}(\bm y)d\bm xd\bm y\\
\geq & -\frac{1}{2\pi}\int_{D}\int_{D}\ln\frac{1}{|\bm x-\bm y|}\varrho_{n,1}(\bm x)\varrho_{n,2}(\bm y)d\bm xd\bm y-\frac{1}{2}\int_{D}\int_{D}h(\bm x,\bm y)\varrho_{n,1}(\bm x)\varrho_{n,1}(\bm y)d\bm xd\bm y\\
&-\frac{1}{2}\int_{D}\int_{D}h(\bm x,\bm y)v_{n,2}(\bm x)v_{n,2}(\bm y)d\bm xd\bm y+\int_{D}\int_{D}h(\bm x,\bm y)v_{n,1}(\bm x)v_{n,2}(\bm y)d\bm xd\bm y,
\end{split}
\end{equation}
or equivalently,
\begin{equation}\label{bk9}
\begin{split}
& \int_{D}\int_{D}G(\bm x,\bm y)\zeta^{+}_{n}(\bm x)\zeta^{-}_{n}(\bm y)d\bm xd\bm y+\frac{1}{2}\int_{D}\int_{D}h(\bm x,\bm y)\zeta^{+}_{n}(\bm x)\zeta^{+}_{n}(\bm y)d\bm xd\bm y\\
&+\frac{1}{2}\int_{D}\int_{D}h(\bm x,\bm y)\zeta^{-}_{n}(\bm x)\zeta^{-}_{n}(\bm y)d\bm xd\bm y\\
\leq & \int_{D}\int_{D}G(\bm x,\bm y)v_{n,1}(\bm x)v_{n,2}(\bm y)d\bm xd\bm y+\frac{1}{2}\int_{D}\int_{D}h(\bm x,\bm y)v_{n,1}(\bm x)v_{n,1}(\bm y)d\bm xd\bm y\\
&+\frac{1}{2}\int_{D}\int_{D}h(\bm x,\bm y)v_{n,2}(\bm x)v_{n,2}(\bm y)d\bm xd\bm y.
\end{split}
\end{equation}
It is clear that the right-hand side of \eqref{bk9} converges to
\[-\kappa_{1}\kappa_{2}G( {\bm x}_{1}, {\bm x}_{2})+\frac{1}{2}\kappa_{1}^{2}h( {\bm x}_{1}, {\bm x}_{1})+\frac{1}{2}\kappa_{2}^{2}h( {\bm x}_{2}, {\bm x}_{2})\]
as $n\to+\infty,$
which is finite. Therefore $ \bar{\bm x}_{1},\bar{\bm x}_{2}\in  D$ and $ \bar{\bm x}_{1}\neq \bar{\bm x}_{2}$, since otherwise the left-hand of \eqref{bk9} would converge  to infinity. 
Passing to the limit $n\to+\infty$, we obtain from \eqref{bk9} that
\begin{equation}\label{bk60}
\begin{split}
 &-\kappa_{1}\kappa_{2}G(\bar{\bm x}_{1},\bar{\bm x}_{2})+\frac{1}{2}\kappa_{1}^{2}h(\bar{\bm x}_{1},\bar{\bm x}_{1})+\frac{1}{2}\kappa_{2}^{2}h(\bar{\bm x}_{2},\bar{\bm x}_{2})\\
 \leq &-\kappa_{1}\kappa_{2}G( {\bm x}_{1}, {\bm x}_{2})+\frac{1}{2}\kappa_{1}^{2}h( {\bm x}_{1}, {\bm x}_{1})+\frac{1}{2}\kappa_{2}^{2}h( {\bm x}_{2}, {\bm x}_{2}),
 \end{split}
 \end{equation}
which is exactly
\begin{equation}\label{habb}
W_{\bm\kappa}(\bar{\bm x}_{1},\bar{\bm x}_{2})\leq W_{\bm\kappa}( {\bm x}_{1}, {\bm x}_{2}).
\end{equation}
Since \eqref{habb} holds for arbitrary $ \bm x_{1},\bm x_{2}\in D$ such that $\bm x_1\neq \bm x_2$, we deduce that 
\[W_{\bm\kappa}(\bar{\bm x}_{1},\bar{\bm x}_{2})=\min_{\bm x_1,\bm x_2\in D, \,\bm x_1\neq \bm x_2}W_{\bm\kappa}( {\bm x}_{1}, {\bm x}_{2}).\]

\end{proof}

\section{Proof of Theorem \ref{thm2}}
In this section we give the proof of Theorem \ref{thm2}.

Fix a sequence $\{\bm\varepsilon_{n}\}$ such that $|\bm\varepsilon_{n}|\to 0$ as $n\to+\infty$, and a sequence $\{\zeta_{n}\}$ such that $\zeta_{n}\in\mathcal M_{\bm\varepsilon_{n}}$ for each $n$.  Denote
\[\xi_{n,i}=\xi_{\bm\varepsilon_n,i},\quad \varrho_{n,i}=\varrho_{\bm\varepsilon_n,i},\quad i=1,2.\]
Then $\xi_{n,i}$ satisfies
\begin{equation}\label{yys1}
\xi_{n,i}\geq 0,\quad \int_{\mathbb R^{2}}\bm x\xi_{n,i}(\bm x)d\bm x=\bm 0,\quad \mbox{\rm supp}(\xi_{n,i})\subset B_{C}(\bm 0),\quad \|\xi_{n,i}\|_{L^{p}(\mathbb R^{2})}\leq C, \quad i=1,2,
\end{equation}
and 
$\varrho_{n,i} $ satisfies
\begin{equation}\label{yys2}
\mbox{\rm supp}(\varrho_{n,i})\subset B_{1}(\bm 0),\quad \|\varrho_{n,i}\|_{L^{p}(\mathbb R^{2})}\leq C,\quad i=1,2,
\end{equation}
where $C>0$ is independent of $\bm\varepsilon.$
 
 Suppose $\varrho_{n,1}, \varrho_{n,2}$ converge weakly in $L^p(\mathbb R^2)$  to $\varrho_1, \varrho_2$ as $n\to+\infty,$ respectively. We need only to show that  $\xi_{n,1}, \xi_{n,2}$ also converge weakly to $\varrho_1, \varrho_2$  in $L^p(\mathbb R^2)$  as $n\to+\infty.$  
 
Let $(\bar{\bm x}_{1}, \bar{\bm x}_{2})$ be a global minimum point of $W_{\bm\kappa}$.
Define $v_{n}=v_{n,1}-v_{n,2},$
\begin{equation}\label{yyss2}
v_{n,1}(\bm x):=\rho_{\bm\varepsilon_n,1}(\bm x-\bar{\bm x}_{1}),\quad v_{n,2}(\bm x):=\rho_{\bm\varepsilon_n,2}(\bm x-\bar{\bm x}_{2}),\end{equation}
 Then for sufficiently large $n$, it holds that $v_{n}\in\mathcal R_{\bm\varepsilon_{n}},$ and thus  
\begin{equation}\label{yyso}
E(\zeta_{n})\geq E(v_{n})\quad\forall\,n.
\end{equation}
In view of the definition of $v_n$, we can write  $E(v_{n})$ as follows:
\begin{equation}\label{gnl1}
\begin{split}
E(v_{n})=&\frac{1}{4\pi}\int_{D}\int_{D}\ln\frac{1}{|\bm x-\bm y|}v_{n,1}(\bm x)v_{n,1}(\bm y)d\bm xd\bm y+\frac{1}{4\pi}\int_{D}\int_{D}\ln\frac{1}{|\bm x-\bm y|}v_{n,2}(\bm x)v_{n,2}(\bm y)d\bm xd\bm y\\
&-\frac{1}{2\pi}\int_{D}\int_{D}\ln\frac{1}{|\bm x-\bm y|}v_{n,1}(\bm x)v_{n,2}(\bm y)d\bm xd\bm y-\frac{1}{2}\int_{D}\int_{D}h(\bm x,\bm y)v_{n,1}(\bm x)v_{n,1}(\bm y)d\bm xd\bm y\\
&-\frac{1}{2}\int_{D}\int_{D}h(\bm x,\bm y)v_{n,2}(\bm x)v_{n,2}(\bm y)d\bm xd\bm y+\int_{D}\int_{D}h(\bm x,\bm y)v_{n,1}(\bm x)v_{n,2}(\bm y)d\bm xd\bm y\\
=&\frac{1}{4\pi}\int_{D}\int_{D}\ln\frac{1}{|\bm x-\bm y|}v_{n,1}(\bm x)v_{n,1}(\bm y)d\bm xd\bm y+\frac{1}{4\pi}\int_{D}\int_{D}\ln\frac{1}{|\bm x-\bm y|}v_{n,2}(\bm x)v_{n,2}(\bm y)d\bm xd\bm y\\
&-\frac{1}{2\pi}\kappa_{n,1}\kappa_{n,2}\ln|\bar{\bm x}_1-\bar{\bm x}_2|-\frac{1}{2}\kappa_{n,1}^2h(\bar{\bm x}_1,\bar{\bm x}_1)-\frac{1}{2}\kappa_{n,2}^2h(\bar{\bm x}_2,\bar{\bm x}_2)-\kappa_{n,1}\kappa_{n,2}h(\bar{\bm x}_1,\bar{\bm x}_2)\\
=&\frac{1}{4\pi}\int_{D}\int_{D}\ln\frac{1}{|\bm x-\bm y|}v_{n,1}(\bm x)v_{n,1}(\bm y)d\bm xd\bm y+\frac{1}{4\pi}\int_{D}\int_{D}\ln\frac{1}{|\bm x-\bm y|}v_{n,2}(\bm x)v_{n,2}(\bm y)d\bm xd\bm y\\
&-W_{\bm\kappa}(\bar{\bm x}_1,\bar{\bm x}_2)+\alpha_n,
\end{split}
\end{equation}
where $\{\alpha_n\}\subset\mathbb R$ such that $\alpha_n\to 0$ as $n\to+\infty.$
In an analogous way, we can use Propositions \ref{prp45} and  \ref{llvv}  to write $E(\zeta_{n})$ as follows:
\begin{equation}\label{gnl2}
\begin{split}
E(\zeta_{n})=&\frac{1}{4\pi}\int_{D}\int_{D}\ln\frac{1}{|\bm x-\bm y|}\zeta^{+}_{n}(\bm x)\zeta^{+}_{n}(\bm y)d\bm xd\bm y+\frac{1}{4\pi}\int_{D}\int_{D}\ln\frac{1}{|\bm x-\bm y|}\zeta^{-}_{n}(\bm x)\zeta^{-}_{n}(\bm y)d\bm xd\bm y\\
&-W_{\bm\kappa}(\bar{\bm x}_1,\bar{\bm x}_2)+\beta_n,
\end{split}
\end{equation}
where $\{\beta_n\}\subset\mathbb R$ such that $\beta_n\to 0$ as $n\to+\infty.$
From  \eqref{yyso}-\eqref{gnl2}, we obtain that 
\begin{equation}\label{gnl3}
\begin{split}
& \int_{D}\int_{D}\ln\frac{1}{|\bm x-\bm y|}v_{n,1}(\bm x)v_{n,1}(\bm y)d\bm xd\bm y+ \int_{D}\int_{D}\ln\frac{1}{|\bm x-\bm y|}v_{n,2}(\bm x)v_{n,2}(\bm y)d\bm xd\bm y\\
\leq & \int_{D}\int_{D}\ln\frac{1}{|\bm x-\bm y|}\zeta^{+}_{n}(\bm x)\zeta^{+}_{n}(\bm y)d\bm xd\bm y+ \int_{D}\int_{D}\ln\frac{1}{|\bm x-\bm y|}\zeta^{-}_{n}(\bm x)\zeta^{-}_{n}(\bm y)d\bm xd\bm y+\gamma_n 
\end{split}
\end{equation}
for some $\{\gamma_n\}\subset\mathbb R$ such that $\gamma_n\to 0$ as $n\to+\infty.$
By a direct calculation, we deduce from \eqref{gnl3} that 
\begin{equation}\label{gnl4}
\begin{split}
&\int_{\mathbb R^{2}}\int_{\mathbb R^{2}}\ln\frac{1}{|\bm x-\bm y|}\varrho_{n,1}(\bm x)\varrho_{n,1}(\bm y)d\bm x d\bm y+\int_{\mathbb R^{2}}\int_{\mathbb R^{2}}\ln\frac{1}{|\bm x-\bm y|}\varrho_{n,2}(\bm x)\varrho_{n,2}(\bm y)d\bm x d\bm y\\
\leq &\int_{\mathbb R^{2}}\int_{\mathbb R^{2}}\ln\frac{1}{|\bm x-\bm y|}\xi_{n,1}(\bm x)\xi_{n,1}(\bm y)d\bm x d\bm y+\int_{\mathbb R^{2}}\int_{\mathbb R^{2}}\ln\frac{1}{|\bm x-\bm y|}\xi_{n,2}(\bm x)\xi_{n,2}(\bm y)d\bm x d\bm y+\gamma_n.
\end{split}
\end{equation}
On the other hand, by Lemma \ref{rri2},
\begin{equation}\label{gnl5}
 \int_{\mathbb R^{2}}\int_{\mathbb R^{2}}\ln\frac{1}{|\bm x-\bm y|}\xi_{n,1}(\bm x)\xi_{n,1}(\bm y)d\bm x d\bm y\leq \int_{\mathbb R^{2}}\int_{\mathbb R^{2}}\ln\frac{1}{|\bm x-\bm y|}\varrho_{n,1}(\bm x)\varrho_{n,1}(\bm y)d\bm x d\bm y,
\end{equation}
\begin{equation}\label{gnl6}
 \int_{\mathbb R^{2}}\int_{\mathbb R^{2}}\ln\frac{1}{|\bm x-\bm y|}\xi_{n,2}(\bm x)\xi_{n,2}(\bm y)d\bm x d\bm y\leq \int_{\mathbb R^{2}}\int_{\mathbb R^{2}}\ln\frac{1}{|\bm x-\bm y|}\varrho_{n,2}(\bm x)\varrho_{n,2}(\bm y)d\bm x d\bm y.
\end{equation}
Suppose there is some subsequence of $\{\xi_n\}$, denoted by $\xi_{n_{j},i},$ such that $\xi_{n_{j},i}\rightharpoonup \xi_{i} $  in $L^p(\mathbb R^2)$ as $j\to+\infty$, $i=1,2$. From \eqref{gnl4}-\eqref{gnl6}, we can pass to the limit $j\to+\infty$ to obtain that
\begin{align*}
 \int_{\mathbb R^{2}}\int_{\mathbb R^{2}}\ln\frac{1}{|\bm x-\bm y|}\xi_{1}(\bm x)\xi_{1}(\bm y)d\bm x d\bm y= \int_{\mathbb R^{2}}\int_{\mathbb R^{2}}\ln\frac{1}{|\bm x-\bm y|}\varrho_{1}(\bm x)\varrho_{1}(\bm y)d\bm x d\bm y,
 \end{align*}
 \begin{align*}
 \int_{\mathbb R^{2}}\int_{\mathbb R^{2}}\ln\frac{1}{|\bm x-\bm y|}\xi_{2}(\bm x)\xi_{2}(\bm y)d\bm x d\bm y= \int_{\mathbb R^{2}}\int_{\mathbb R^{2}}\ln\frac{1}{|\bm x-\bm y|}\varrho_{2}(\bm x)\varrho_{2}(\bm y)d\bm x d\bm y.
\end{align*}
Applying Lemma \ref{bgu} (note that the assumptions in Lemma \ref{bgu} are satisfied by \eqref{yys1}), we deduce that $\xi_{i}=\varrho_{i},$ $i=1,2$. Hence $\xi_{n_{j},i}\rightharpoonup \varrho_{i}$  in $L^p(\mathbb R^2)$ as $j\to+\infty$, $i=1,2$.  Arguing by contradiction, we can further show that weak convergence actually holds for the whole sequence $\{\xi_{n,i}\},$ $i=1,2$.
Hence item (i) of Theorem \ref{thm2} has been proved.

If additionally $\varrho_{n,i}\to \varrho_{i}$ in $L^{p}(\mathbb R^{2})$ as $n\to+\infty$, $i=1,2,$ then 
\[\lim_{n\to+\infty}\|\varrho_{n,i}\|_{L^{p}(\mathbb R^{2})}= \|\varrho_{i}\|_{L^{p}(\mathbb R^{2})}, \quad i=1,2,\]
which implies that
\[\lim_{n\to+\infty}\|\xi_{n,i}\|_{L^{p}(\mathbb R^{2})}= \|\varrho_{i}\|_{L^{p}(\mathbb R^{2})}\quad i=1,2.\] 
In combination with the weak convergence proved above,  we deduce by uniform convexity that $\xi_{n,i}\to \varrho_{i}$  in $L^p(\mathbb R^2)$ as $n\to+\infty$, $i=1,2$, which verifies  item (ii) of Theorem \ref{thm2}.

\bigskip
 \noindent{\bf Acknowledgements:}
{G. Wang was supported by National Natural Science Foundation of China (12001135) and China Postdoctoral Science Foundation (2019M661261, 2021T140163). B. Zuo was supported by National Natural Science Foundation of China (12101154).}

 

\phantom{s}
 \thispagestyle{empty}

\end{document}